\documentclass[twoside,a4paper,reqno,11pt]{amsart} 
\usepackage[top=28mm,right=28mm,bottom=28mm,left=28mm]{geometry}
\usepackage{microtype}
\usepackage{amsfonts, amsmath, amssymb, amsbsy, mathrsfs, bm, latexsym, stmaryrd, hyperref, array, enumitem, textcomp, url}
\usepackage[table]{xcolor}

\newcommand{\PSU}{\operatorname{PSU}}

\newcommand{\alt}{\operatorname{Alt}}

\newcommand{\PSL}{\operatorname{PSL}}

\newcommand{\SL}{\operatorname{SL}}
\newcommand{\Sp}{\operatorname{Sp}}

\newcommand{\N}{\mathrm{N}}

\makeatletter
\newcommand{\imod}[1]{\allowbreak\mkern4mu({\operator@font mod}\,\,#1)}
\makeatother

\theoremstyle{plain}

\newtheorem{thm}{Theorem}[section] 
\newtheorem{lem}[thm]{Lemma}
\newtheorem{prop}[thm]{Proposition} 

\newtheorem{cor}[thm]{Corollary} 

\newtheorem*{theorem*}{Theorem} 
\newtheorem*{conj*}{Conjecture}

\theoremstyle{definition}

\begin{document}

\title[Generation of finite groups]{Finite groups can be generated by a $\pi$-subgroup \\ and a $\pi'$-subgroup}

\author{Thomas Breuer}
\address{T. Breuer, Lehrstuhl f\"ur Algebra und Zahlentheorie, RWTH Aachen, Germany}
\email{sam@math.rwth-aachen.de}
  
\author{Robert M. Guralnick}
\address{R.M. Guralnick, Department of Mathematics, University of Southern California, Los Angeles, CA 90089-2532, USA}
\email{guralnic@usc.edu}

\date{\today} 

\thanks{
The first author was funded by the Deutsche Forschungsgemeinschaft
(DFG, German Research Foundation) -- Project-ID 286237555 -- TRR 195.
The second author was partially supported by the NSF
grant DMS-1901595 and a Simons Foundation Fellowship 609771.
We thank Dan Haran for the question.  We thank both Dan Haran 
and Alex Lubotzky for pointing out the connections with free profinite groups.
We also thank the referee for their careful reading and many very helpful comments on earlier
versions.} 

\begin{abstract}   
Answering a question of Dan Haran and generalizing some results of Aschbacher-Guralnick and Suzuki,
we prove that given a set of  primes $\pi$, any finite group can be generated by a $\pi$-subgroup and a 
$\pi'$-subgroup.   This gives a new free product description of a free profinite group. 
\end{abstract}

\dedicatory{For Moshe Jarden on the occasion of his 80th birthday} 

\maketitle

%\setcounter{tocdepth}{1}
%\tableofcontents

\section{Introduction} \label{s:intro}

Recall that a subgroup $H$ of $G$ is called intravariant if for any automorphism $a$ of $G$,
$a(H)$ and $H$ are conjugate in $G$.  There are many examples of such subgroups (including
Sylow subgroups).  See \cite{GH} for many results in that direction for simple groups.

If $\pi$ is a set of primes, we let $\pi'$ be the complementary set of primes. 

Our first main result is the following (answering a generalization of a question of Haran). 

\begin{thm} \label{t:main1}  Let $\pi$ be a set of primes and $G$ a finite group.   Then there
exist intravariant subgroups $P$ and $R$ of $G$ such that $G = \langle P, R \rangle$ with
$P$ a $\pi$-subgroup and $R$ a $\pi'$-subgroup.
\end{thm}  

Of course, by the Feit-Thompson theorem \cite{FT}, any odd order subgroup is solvable.  
In the case $\pi=\{2\}$, this gives:

\begin{cor} \label{c:ag}  Let $G$ be a finite group.  Then $G$ can be generated by a Sylow $2$-subgroup
and an intravariant (solvable) group of odd order.
\end{cor}

In \cite[Thm. A]{AG}, it is shown that every finite group can be generated by two conjugate solvable groups and so
by a cyclic subgroup and a solvable group.    This was generalized in \cite{Su} by  noting  that we can take the solvable group to be
intravariant.   We note that the proof of Theorem \ref{t:main1}  here would give another proof of the Aschbacher-Guralnick and Suzuki results. 

A complementary result to the previous corollary is obtained in \cite[Cor. 2]{Gu}.   This asserts that any
finite group be generated by a Sylow $2$-subgroup and a solvable subgroup containing a
Sylow $2$-subgroup.    The conclusion
that the subgroup is intravariant was not stated but the proof can be modified to yield that conclusion.  

 It was also noted in \cite{AG} that not every finite group can be generated by two nilpotent groups
(and this was generalized in \cite{CH} to show that not every finite  group can be generated by $m$ nilpotent
subgroups for any fixed $m$).   

Dan Haran and Alex Lubotzky pointed  out that using the proof of Theorem \ref{t:main1}
and  Iwasawa's criterion for freeness, one can show:

\begin{thm} \label{t:main2}   Let $\pi$ be a set of primes.   Let $A$ be the free pro-$\pi$ group on countably
many generators and $B$ be the free pro-$\pi'$ group on countably
many generators.  Then $A*B$ is a free profinite group.
\end{thm}

Since a free profinite group on countably many generators surjects onto any countably generated profinite group, an
immediate consequence of Theorem \ref{t:main2}  is the following.  

\begin{cor} \label{c:pro}  Let $G$ be a countably generated profinite group.
\begin{enumerate} 
\item $G$ can be generated by a  pro-$\pi$-subgroup and a pro-$\pi'$-subgroup.
\item $G$ can be generated by a pro-$2$-subgroup and a pro-solvable subgroup. 
\end{enumerate} 
\end{cor} 

Similarly, using Lemmas \ref{l:lift1}, \ref{l:lift2} and \ref{l:lift4}, one sees that:

\begin{thm} \label{t:main3}  Let $S$ be the free prosolvable group on countably many generators, then $S*S$ is a free profinite group.   If $\pi$ is a set of primes and either $2 \in {\pi}$ or $\{3,5\} \subseteq \pi$ and $A$ is the free pro-$\pi$-group on countably
many generators, then $A*S$ is a free profinite group.
\end{thm} 

The condition on $\pi$ in the previous result is required to ensure that every finite nonabelian simple group
has order divisible by some prime in $\pi$.

As far as we know,  these are the first examples known of a free product description of  a free profinite group 
with neither factor free.   See \cite{HL} for an example where  one factor of the free product is projective but not
free and the other is free.  

One can translate the result to field extensions.  We only state it in the case of finite extensions. 

\begin{cor}  \label{c:field}  Let $\pi$ be a set primes.   Let $K/F$ be a finite Galois extension.   Then there exist
subfields $K_1$ and $K_2$ with $F=K_1 \cap K_2$ such that $[K:K_1]$ is a $\pi$-number and
$[K:K_2]$ is a $\pi'$-number. 
\end{cor}

There are two key steps in the proofs.  The first is  a stronger result for simple groups.   The classification of finite simple
groups is required for our proof of this. 

\begin{thm} \label{t:simple}  Let $L$ be a finite   simple group.  There exists a prime $r$ (depending
on $L$)  such that given any prime $t$ dividing $|L|$,  $L$ can be generated
by a Sylow $r$-subgroup and a Sylow $t$-subgroup. 
\end{thm}

The other step is a lifting theorem for homomorphisms onto a $\pi$-group or solvable group.

We conjecture that in fact we can take any prime $r$ dividing $|L|$ above or equivalently
that if $r$ and $s$ are primes dividing the order of the finite simple group $L$, then
$L$ can be generated by a Sylow $r$-subgroup and a Sylow $s$-subgroup.  It is proved
in \cite[Cor. 18]{BGG} that for a fixed $r,s$ there are at most finitely many exceptions.  

In the next section, we prove Theorem \ref{t:simple}.   We then prove the lifting
results we need and in  the sections \ref{s:main1} and \ref{s:free}  we prove the main results. 

We conclude by showing that the conjecture is true for alternating and sporadic groups.

\section{Simple Groups}

\begin{prop} \label{p:alt}   Let $S = \alt_n, n \ge 5$.   If $ t \le n$ is prime, then $S$ can be 
generated by a Sylow $2$-subgroup and a Sylow $t$-subgroup.
\end{prop} 

\begin{proof}  For $n \le 8$, this is an easy exercise.  Suppose that $n > 8$ and is odd.
Let $R$ be a Sylow $2$-subgroup of $S$.   Note that the orbits of $R$ have distinct
sizes and there are at most  $\log_2 (n+1)$ orbits.   Choose $x$ a $t$-element so that $x$
has a cycle of length $t^a \ge \sqrt{n}$ for $t$ odd and at least $n/2$ if $t=2$.  In both
cases, the length of the cycle is greater than the number of orbits of $R$.  Thus, we can
choose a conjugate of $x$ so that an orbit of $x$ intersects every orbit of $R$.
Then $J:= \langle R, x \rangle$ is transitive.   Since $n$ is odd and $R \le J$,  $J$
is primitive.  Since $n > 8$, the only primitive subgroup of $\alt_n$ containing an element moving
$4$ points is $\alt_n$ \cite[13.5]{Wi}, whence the result.

Suppose that $n > 9$ is even.   By the odd case, $\alt_{n-1} = \langle R_0, T \rangle$ for some
  $t$-subgroup $T$  and  Sylow $2$-subgroup  $R_0$.  Let $R$ be a Sylow $2$-subgroup  of $S$ properly
  containing $R_0$.   Then $\langle R, T \rangle$ properly contains  $\langle R_0, T \rangle = \alt_{n-1}$,
  whence $\langle R, T \rangle = S$. 
\end{proof} 

For the remaining simple groups, we can prove a stronger generation result.   The analogous result
for alternating groups fails (for example, consider $\alt_{15}$ with $x$ a $3$-cycle or an involution moving
only four points).    Note that the next result also holds for quasisimple groups as long as $x$ is noncentral.

\begin{prop} \label{p:spor}
Let $S$ be a sporadic simple group
and let $P$ be a Sylow $2$-subgroup of $S$.
If $1 \ne x \in S$, then
$S = \langle P, x^g\rangle$ for some $g \in S$.
\end{prop} 

\begin{proof}
Let $S$ be a sporadic simple group,
fix a Sylow $2$-subgroup $P$ of $S$,
and let $x$ be a nonidentity element in $S$.
We use known information about maximal subgroups of $S$
to show that $x^S$ is not a subset of the union
of those maximal subgroups in $S$ that contain $P$.

Let $M$ be a maximal subgroup of $S$ with the property $P \le M$.
The number of $S$-conjugates of $M$ that contain $P$
is equal to $|\N_S(P)|/|\N_M(P)| \le [\N_S(P):P]$,
thus these subgroups can contain at most
$[\N_S(P):P] |x^S \cap M|$ elements from the class $x^S$.

Thus the number of elements in $x^S$ that generate a proper subgroup of $S$ 
together with $P$ is bounded from above by
$[\N_S(P):P] \sum_M |x^S \cap M|$,
where the sum is taken over representatives $M$ of conjugacy classes
of maximal subgroups of odd index in $S$.

Let $1_M^S$ denote the permutation character of $S$ on the cosets of $M$.
We have $|x^S \cap M| = |x^S| 1_M^S(x) / 1_M^S(1)$.
Hence we are done when we show that
\[
   [\N_S(P):P] \sum_M 1_M^S(x) / 1_M^S(1) < 1 \tag{$\ast$}\label{inequality}
\]
holds.

The numbers $[N_S(P):P]$ can be read off from \cite[Table I]{Wil98}.
For all sporadic simple groups $S$ except the Monster group,
the permutation characters $1_M^S$ can be computed from the data about
maximal subgroups contained in the library of character tables \cite{CTblLib}
of the \textsc{GAP} computer algebra system \cite{GAP}.
The Monster group is known to contain exactly five classes
of maximal subgroups of odd index, of the structures
$2^{1+24}.Co_1$, $2^{10+16}.O_{10}^+(2)$, $2^{2+11+22}.(M_{24} \times S_3)$,
$2^{5+10+20}.(S_3 \times L_5(2))$, $[2^{39}].(L_3(2) \times 3S_6)$,
and the corresponding permutation characters are known,
see \cite{PermChars}.

We are lucky, the crude upper bound from (\ref{inequality}) is smaller than $1$
for each $S$ and $x$ in question.
(In fact, the left hand side is much smaller than needed;
the maximum of the left hand side is $3/5$,
it is attained for $x$ in the class \texttt{2B} of the group $J_2$.)
\end{proof}

We next need an elementary general result.

\begin{lem} \label{l:twosub}   Let $G$ be a group with proper subgroups $H_1$ and $H_2$.
Let $C$ be a   conjugacy class of $G = \langle C \rangle$.     Then $C$ is  not contained
in $H_1 \cup H_2$.
\end{lem}

\begin{proof}  Let $C_i = C \cap H_i$ and suppose that $C = C_1 \cup C_2$. 
Let $D= C_1 \cap C_2$.   Since $H_i$ is a proper subgroup, $C \ne C_i$.  
This implies that each $C_i$ properly contains $D$.

Note that
$H_2$ normalizes the subgroup $L$  generated by $C_1':=C_1 \setminus D$ (since 
$C$ is the disjoint union of $C_1'$ and $C_2$).

Then $C \subset H_2L$, whence $G=H_2L$ and $L$ is normal in $G$.
Since $L$ intersects $C$, this implies that $C \subset L$ and so $L=G$.
Since $L \le H_1$, this is a contradiction and the result holds.  
\end{proof}

An immediate consequence is the following.

\begin{cor}\label{c:twosub}  Let $G$ be a  simple group with $C$ a nontrivial
conjugacy class of $G$.  If $H_1$ and $H_2$ are proper subgroups of $G$,
then $C$ is not contained in $H_1 \cup H_2$.
\end{cor}

Finally we turn to the case of finite groups of Lie type.   We will typically work
with the simply connected groups and rather than use the prime $2$ as the base
prime (which almost certainly would work), we use the prime $p$ that is the characteristic
of the group.   It is proved in \cite[Thm. A]{Gu} that any nonabelian finite simple group can be
generated by an involution and a Sylow $2$-subgroup.

\begin{prop}  \label{p:lie}  Let $S$ be a finite simply connected quasisimple group of Lie type
in characteristic $p$.   Let $U$ be a Sylow $p$-subgroup of $G$.  If $x \in S$ is a noncentral
element of $S$, then $S = \langle U, x^g \rangle$ for some $g \in S$. 
\end{prop}

\begin{proof}   By a lemma of Tits \cite[1.6]{Se},    any maximal subgroup
containing $U$ is a parabolic subgroup and there is precisely one parabolic
subgroup  containing $U$ in each conjugacy class of maximal parabolic subgroups.
In particular, 
$B=N_G(U)$, the Borel subgroup containing $U$ is contained in every maximal
subgroup containing $U$. 

If $S$ has twisted rank at most $2$, then $B$  is contained in at most two maximal
subgroups and Corollary \ref{c:twosub} implies that some conjugate of $x$ is in neither
maximal subgroup. 
 So we may assume that the rank is at least $3$.  
 
 We first consider the case that $S$ is a classical group with natural module $V$.     Let $\mathcal{F}$ be
 a maximal flag fixed by $B$.   First assume
 that $S$ is not an orthogonal group of $+$ type.     
 Let $W$ be the unique maximal element in the flag.  So $W$ is a maximal totally singular space
 (or hyperplane in the case $S=\SL$).  
 We may assume that $x$ does not fix $W$.   Thus, there exists a line $L \subset W$ with $xL$ not
 contained in $W$.   Let $\mathcal{F'}$  be a maximal flag with $1$-space $L$ and $W$ in the flag.
 Then $x$ does not fix any subspace in the flag and so $x$ is not contained in any maximal parabolic
 subgroup containing $B$ and the result follows.

Suppose that $G= \Omega_{2m}^+(q)$ with $m \ge 3$.   If $m=3$,  the results follows from the
result for $\SL_4(q)$.   So $m \ge 4$.     Let $\mathcal{F}$ be a maximal flag corresponding to
$B$.   Let $W_1$ and $W_2$ be the two maximal singular subspaces in the flag.   By 
Corollary \ref{c:twosub},  we may assume that $x$ fixes neither $W_1$ nor $W_2$.  Since
$W_1$ and $W_2$ are not in the same $G$-orbit,  $xW_i \ne W_j$ for any $i,j$. 
Let $W = W_1 \cap W_2$.  We claim that $xW$ is not contained in $W_1 \cup W_2$.
Since $xW$ is not the union of two proper subspaces, it suffices to show that $xW$ is not contained
in $W_1$ or $W_2$. 
Note that $xW$ is contained in a unique maximal totally singular subspace of each type
(because $W^{\perp}/W$ is a $2$-dimensional nondegenerate space and has precisely
two singular lines). 
Since $xW \subset xW_i \ne W_i$,  $xW$ is not contained in $W_i$ and the claim follows.
Thus we can choose a line $L \subset W$
so that $xL$ is contained in neither $W_1$ or $W_2$.   Consider a maximal flag
starting with $L$ and with maximal elements $W_1$ and $W_2$.  Thus, $x$ fixes
no subspace in this flag and so $x$ and the Borel subgroup corresponding to that
flag generate $G$. 
 
Finally suppose that $S$ is an exceptional group of  twisted rank $n$.  As above, we may
assume that $n > 2$.    Let $P_1, \ldots, P_n$ be
the distinct maximal parabolic subgroups containing $U$.  Let $C = x^S$ for $x$
a noncentral element of $S$.    It follows by \cite[Thm. 1]{LLS} that
$$
\sum_{i=1}^n  \frac{|C \cap P_i|}{|C|} < 1.
$$
Thus,  $C \ne \cup_{i=1}^n (C \cap P_i)$ and the result follows.
\end{proof}

It remains to consider the special cases where $S$ is a simply connected group but $S/Z(S)$ is not
simple.   If $S$ is solvable, then we can ignore the case. This  leaves the groups $\Sp_4(2)$, $G_2(2)$,
 ${^2}G_2(3)$ and ${^2}F_4(2)$.   The first three cases have socle $\alt_6, \PSU_3(3)$ and $\PSL_2(8)$
 and we have already proved the result for these groups.   In the final case, the socle
 is the derived subgroup of index $2$ but  precisely the same argument as above applies
 (i.e. there are only two maximal subgroups containing a Sylow $2$-subgroup which are parabolic
 subgroups intersected with the derived subgroup). 
This completes the proof of Theorem \ref{t:simple}.

\section{Lifting Results}

Let $f:G \rightarrow H$ be a surjective homomorphism of finite groups with $H$ a $\pi$-group.

We first note the well known result which follows easily from the Schur-Zassenhaus theorem \cite[18.1]{As}
and asserts the existence of $M$ in the case that $\ker(f)$ is a $\pi'$-group.  

\begin{lem} \label{l:lift0}  There exists a $\pi$-subgroup $M$ with $f(M)=H$. 
\end{lem}  

 We need a slightly stronger version of  this result for our application.  We show that we can take 
 $M$ to be an intravariant $\pi$-group.   This depends on the conjugacy of $\pi$-complements in the 
 Schur-Zassenhaus result (and so depends on the Feit-Thompson result that groups of odd order 
 are solvable).

 \begin{lem} \label{l:lift1}    Let $f:G \rightarrow H$ be a surjective homomorphism of finite groups with 
 $H$ a $\pi$-group.   Then there exists an intravariant $\pi$-subgroup $M$ of $G$ such
 that $f(M)=H$.  Moreover, we may assume $M \ge O_{\pi}(G)$. 
 \end{lem}
 
 \begin{proof}  Note that if we show that $M$ exists, we can always assume that $M \ge O_{\pi}(G)$. 
 
 Let $A=O_{\pi'}(G)$.  If $A \ne 1$,  then by induction there exists an intravariant $\pi$-subgroup 
 $L/A$ of $G/A$ with $f(L)=H$.  Note that $L$ is intravariant.  By the Schur-Zassenhaus theorem,
 $L=AM$ where $M$ is a $\pi$-complement to $A$ and is unique up to conjugacy.  Since $L$ is intravariant,
 so is $M$ and the result follows.
 
 So we may assume that $A=1$.  Let $B=O_{\pi}(G)$.  If $B \ne 1$,   by induction, there exists an intravariant
 $\pi$-subgroup $L/B$ of $G/B$ with $f(L/B) = H/f(B)$ and so $f(L)=H$.  Clearly $L$ is intravariant.  
 So we may assume that $B=1$ as well.   In particular, $F(G)=1$ and $1 \ne E = E(G)$.     
 Moreover, any component of $E$ is divisible by primes in $\pi$ and $\pi'$ and so $E \le \ker(f)$. 
 
 Let $1 \ne R$ be a Sylow $r$-subgroup of $E$.   So $R$ is intravariant and thus so is $N_G(R)$.
 By the Frattini argument,  $G=EN_G(R)$ and $N_G(R)$ is proper in $G$ since $O_r(G)=1$.   By induction, there exists an intravariant
 $\pi$-subgroup of $N_G(R)$    which surjects onto $H$ and the result follows.  
 \end{proof}

The corresponding result for solvability is the following:

\begin{lem} \label{l:lift2}  Let $f:G \rightarrow H$ be a surjective homomorphism with $H$ solvable.
 Let $A$ be the solvable radical of $G$.  
Then there exists $M \le G$ with $M \ge A$ solvable and intravariant with $f(M)=H$.
\end{lem}

\begin{proof}    If $A \ne 1$, the result holds by induction for the map from $G/A \rightarrow H/f(A)$.
Thus, there exists an intravariant subgroup $L$ containing $A$ such that $L/A$ surjects onto $H/f(A)$, 
whence $f(L)=H$ and the result follows. 

If $A =1$, then $1 \ne E = E(G)$ is a direct product of nonabelian simple groups.    Thus, $f(E)=1$.  
Let $1 \ne R$ be a Sylow $r$-subgroup of $E$ for some prime $r$.   Then $G = EN_G(R)$.
Since $E$ is characteristic in $G$,  $N_G(R)$ is intravariant.  By induction, there
exists an intravariant solvable subgroup $M \le N_G(R)$ with $f(M)=H$ and the result follows.  
 \end{proof} 

\section{Proof of  Theorem \ref{t:main1}  } \label{s:main1}

We first prove an elementary result relating to intravariance of subgroups.

\begin{lem} \label{l:intra1}   Let $G$ be a group with $X$ an intravariant subgroup of $G$.
If $Q \le A \le X$,  $A$ is characteristic in $G$ and $Q$ is intravariant in $A$, then
$Q$ and $N_X(Q)$ are intravariant in $G$.
\end{lem} 

\begin{proof}  Since $Q$ is intravariant in $A$ and $A$ is normal
in $X$,  $X=AN_X(Q)$.  Let $a$ be an automorphism of $G$.  By modifying $a$ by inner automorphism,
we may assume that $a$ leaves $X$ invariant.   
Since $Q^a$ and $Q$
are conjugate in $A$, we may modify $a$ by an inner automorphism and assume that
$a$ normalizes $Q$ and $X$ and so also $N_X(Q)$ and the result follows.
\end{proof}
 
We now prove Theorem \ref{t:main1}.  Fix a set of primes $\pi$ and let $G$ be a counterexample of minimal order.  \\

\noindent
Step 1.   $O_{\pi}(G) = 1$.  \\

If not, let $1 \ne A$ be a minimal characteristic subgroup of $G$ with $A$ a $\pi$-group. 
By minimality,  $G/A = \langle P/A,  R/A \rangle$ where $P$ and $R$ are subgroups of $G$
with $P/A$ and $R/A$ intravariant in $G/A$ and $P/A$ a $\pi$-group and $R/A$ a $\pi'$-group.
Since $A$ is characteristic in $G$,  $P$ and $R$ are intravariant in $G$.   
Clearly $P$ is a $\pi$-group.   By the Schur-Zassenhaus theorem,   $R = AQ$ where $Q$ is
a $\pi'$-group.   Moreover, the conjugacy part of the Schur-Zassenhaus theorem (which requires
the Feit-Thompson theorem), implies that $Q$ is intravariant in $G$.    Clearly, $G= \langle P, Q \rangle$. \\

Step 2.  $O_{\pi'}(G)=1$.\\

Interchange $\pi$ and $\pi'$ in Step 1. \\

Step 3.  Completion of the proof. \\

Let $A$ be a minimal characteristic subgroup of $G$.  
 By Steps 1 and 2,  we may assume that $F(G)=1$ and  $A = L \times \ldots \times L$
where $L$ is a nonabelian simple group of order divisible by primes in $\pi$ and $\pi'$.

 By Theorem \ref{t:simple}, 
$A = \langle Q, M \rangle$ where $Q$ is a Sylow $p$-subgroup of $A$ and $M$ is a Sylow $r$-subgroup
of $A$ for some $r \ne p$ where we may take $p \in \pi$ and $r \in \pi'$.  

By induction,  $G/A = \langle X/A, Y/A  \rangle$ where $X/A$ and $Y/A$ are intravariant subgroups
of $G/A$ with $X/A$ a $\pi$-group and $Y/A$ a $\pi'$-group.   Since $A$ is characteristic in $G$,
this implies that $X$ and $Y$ are intravariant subgroups of $G$.    Since $X/A$ is a $\pi$-group, 
$A$ is the subgroup of $X$ generated by all $\pi'$-elements and so is characteristic in $X$.
Similarly, $A$ is characteristic in $Y$.   

By the Frattini argument,  $X=AN_X(Q)$.  By Lemma \ref{l:intra1}, $N_X(Q)$ is intravariant in
$G$.   By Lemma \ref{l:lift1}, there exists an intravariant 
$\pi$-subgroup $X_1$ of $N_X(Q)$ with  $X=AX_1$, 
 $Q \le O_{\pi}(N_X(Q)) \le X_1$ and $X_1$ intravariant.       

The same argument replacing $Q$ by $M$ and $\pi$ by $\pi'$ allows to assert there exist
subgroups $X_1 \le X$ and $Y_1 \le Y$ such that:
\begin{enumerate}
\item $Q \le X_1$ and $X_1$ is  an intravariant $\pi$-subgroup of $G$ with $X=X_1A$, and
\item $M \le Y_1$  and $Y_1$ is an intravariant $\pi'$-subgroup of $G$ with $Y=Y_1A$.
\end{enumerate}

Since $A = \langle Q, M \rangle$,  $A \le \langle X_1, Y_1 \rangle$.   Since  $G/A = \langle X_1A/A, Y_1A/A \rangle$,
it follows that  $G= \langle X_1, Y_1 \rangle$ as required. 

\section{Free profinite groups} \label{s:free} 

We now prove Theorem \ref{t:main2}.  The proof is essentially the same as in the previous section but
we need to use Iwasawa's criterion \cite[p. 84]{Ri} to deduce freeness.  

In order to apply this criterion for Theorem \ref{t:main2}, we first need to note the following:

\begin{lem} \label{l:lift3}  Let $G$ be a finite group and $A$ a nontrivial normal subgroup.   If $G/A = \langle X/A, Y/A \rangle$
where $X/A$ is a $\pi$-group and $Y/A$ is a $\pi'$-group, then there exist $X_1 \le X$ a $\pi$-group  and $Y_1 \le Y$
a $\pi'$-group so that $G = \langle X_1, Y_1 \rangle$ and $X_1A=X$ and $Y_1A = Y$. 
\end{lem} 

\begin{proof}   Clearly, we may 
assume that $A$ is a minimal normal subgroup of $G$.   The three steps in the previous section give
the result (there we assumed that $A$ was characteristic to deduce that the
subgroups were intravariant but the proof is identical).  
\end{proof} 

This lemma can be reformulated in the following way.   Let  $F(\pi)$ (respectively $F(\pi')$)  denote the free pro-$\pi$ group
(respectively the free pro-$\pi'$ group) 
on countably many generators and let $F=F(\pi)*F(\pi')$.   Let $G$ be a finite group and $A$ a normal subgroup of $G$.
Suppose that $\alpha:F \rightarrow G/A$ is surjective.   Let $X/A = \alpha(F(\pi))$ and $Y/A = \alpha(F(\pi'))$.  Thus
$G/A = \langle X/A, Y/A \rangle$.   Since $F(\pi)$ and $F(\pi')$ are free pro-$\pi$ (free pro-$\pi'$), it follows
by Lemma \ref{l:lift3}  that  there exist lifts of $\alpha$,
 $\beta: F(\pi) \rightarrow  X$ and
$\gamma: F(\pi') \rightarrow Y$ so that the images of $\beta$ and $\gamma$ generate $G$.  Thus, $\beta*\gamma$
is a lift of $\alpha$  from $F$ onto $G$.    Iwasawa's criterion implies that $F$ is a free profinite group.  Thus, Theorem \ref{t:main2}
follows. 

If $\pi$ and $\pi'$ are both nonempty, $F(\pi)$ and $F(\pi')$ are both projective profinite groups but neither is free. 

Similarly, one proves the analogous statement with $X/A$ and $Y/A$ solvable using 
Lemma \ref{l:lift2}.   Note by taking $\pi=\{2\}$,  Lemma \ref{l:lift3} already proves a solvable lifting theorem.

\begin{lem} \label{l:lift4}  Let $G$ be a finite group and $A$ a nontrivial normal subgroup.   If $G/A = \langle X/A, Y/A \rangle$
where $X/A$ and $Y/A$ are solvable,  then there exist $X_1 \le X$ and $Y_1 \le Y$ with $X_1$ and $Y_1$ both solvable 
 so that $G = \langle X_1, Y_1 \rangle$ and $X_1A=X$ and $Y_1A = Y$. 
\end{lem} 

Essentially the same proof gives a lifting result for pairs of subgroups where $X$ is a $\pi$-group and $Y$ is a solvable
group as long as $\pi$ contains a prime dividing the  order of the simple groups occurring.  In particular if $2 \in  \pi$,
this is fine.  Since every nonabelian finite simple group has order divisible by $3$ except for the Suzuki groups,
which have order divisible by $5$,   the lifting lemma goes through if $\{3,5\} \subseteq \pi$ as well.  

Now Theorem \ref{t:main3} follows.

\section{Alternating Groups}

In this section, we prove:

\begin{thm} \label{t:altsyl}   Let $n \ge 5$ and let $p$ and $q$ be primes with $p \le q \le n$.
Then $G:=\alt_n$ can be generated by a Sylow $p$-subgroup and a Sylow $q$-subgroup.
\end{thm}

\begin{proof}   We have already proved this with $p=2$.  So assume that $p$ is odd.   Set $G = \alt_n$.    
We induct on $n$.  If $n \le 11$, this is easily checked.  So assume $n > 11$.    

First suppose that  $p$ or $q$ properly divides $n$.   Then by induction,  $\alt_{n-1} = \langle P_0, Q_0 \rangle$
where $P_0$ is a Sylow $p$-subgroup and $Q_0$ is a Sylow $q$-subgroup of $\alt_{n-1}$.  Let $P \ge P_0$
be a Sylow $p$-subgroup of $G$ and $Q \ge Q_0$ a Sylow $q$-subgroup.   Then either $P \ne P_0$
or $Q \ne Q_0$ and so $\langle P, Q \rangle$ properly contains $\alt_{n-1}$ and the result holds.   

Next suppose that $q = n$.   Let $Q = \langle y \rangle$ be a Sylow $q$-subgroup.  It follows by a result of
Manning \cite[13.9]{Wi} that if $P$ is any Sylow $p$-group, then unless $p = q - e$ with $e=0$ or $2$ that $\langle P, Q \rangle = G$ (since
it is primitive and contains a $p$-cycle).   If $p=q - e$, we can choose $x$ a $p$-cycle. so that $xy$ is a $3$-cycle and so 
if $P=\langle x \rangle$, then $\langle P, Q \rangle$ is a primitive group containing a $3$-cycle, whence $G=\langle P, Q \rangle$.

If $n - 3 \le p \le q < n$, then we can choose a $p$-cycle $x$ and a $q$-cycle $y$ so that $\langle x, y \rangle$ is transitive and
$xy$ is a cycle of length $3$, $5$ or $7$.    Again the result of Manning, $G=\langle x, y \rangle$ for $n > 10$. 

Now suppose that $\gcd(pq, n)=1$ and $ p \le n -4$.  By Manning's result it suffices to show that we can choose Sylow
subgroups $P$ and $Q$ so that $\langle P, Q \rangle$ is primitive.  

We first show that we can choose $P$ and $Q$ so that $\langle P, Q \rangle$ is transitive.  Let $a$ be the largest power of
$p$ less than $n$ and $b$ the largest power of $q$ is that less than $n$.   Assume that $b \ge a$ (if not the interchange
the roles of $a$ and $b$).  Write $n  = \alpha a + \beta$ with $1 \le \alpha < p \le a$ and $0 \le \beta < a$.  Let $x$ be an element
with $\alpha$ orbits of size $a$ and $\beta$ fixed points.  If $b \ge \alpha + \beta$,  then let $y$ be a $b$-cycle that intersects
each orbit of $y$.   Then $\langle x, y \rangle$ is transitive.   Note that $\alpha + \beta < 2a \le 2b$.   Thus, unless $b > n/2$, we can choose
$y$ a product of two disjoint cycles of length $b$ so that each orbit of $y$ intersects one of the two cycles of $y$ and moreover
at least one orbit of $x$ intersects both cycles.  Then again $\langle x, y \rangle$ is transitive.  

The remaining case is when $\alpha + \beta > b > n/2$.  We claim this is not possible.    If $\alpha > n/4$, then $a \le 3$ and so $p=a=3$
and $n \le 8$.   Otherwise, $a > \beta > n/4$ and so $\alpha \le 3$.    Note that $n > (\alpha + 1) \beta$ and $\alpha \le 2$.   
If $\alpha = 2$, then $\beta < n/3$.   Thus $n/2 < \alpha + \beta < 2 + n/3$, whence $n < 12$.
If $\alpha =1 $, then $b \ge a > n/2$ and so $b \ge \beta +1$ and the result follows. 

Let $P$ be a Sylow $p$-subgroup containing $x$ and $Q$ a Sylow $q$-subgroup containing $y$.
Then $H: = \langle P, Q \rangle$ is transitive. We claim that $H$ is primitive and then again by Manning 
$G=H$.   Suppose that $H$ is imprimitive.   Note that $P$ and $Q$ are each generated by elements which
are cycles.  Thus some $1 \ne g \in P \cup Q$ which is a cycle must move some block of imprimitivity (since $H$
is transitive).   This implies that points moved by $g$ must be a union of blocks.   Thus the size of a block must divide
a power of $p$ or a  power $q$.    This implies
that   $\gcd(pq, n) \ne 1$, a contradiction. 
\end{proof}

\section{Sporadic Groups}  

\begin{thm} \label{t:sporsyl}   Let $G$ be a sporadic simple group and let $p \le q$ be primes 
each dividing $|G|$.   Then $G$ can be generated by a Sylow $p$-subgroup and a Sylow $q$-subgroup.
\end{thm}

\begin{proof}
Let us first assume that $G$ is not the Monster.

As a first step, we generalize the approach from the proof of
Proposition~\ref{p:spor},
in order to check for which prime divisors $p$ of $|G|$
and for which nontrivial conjugacy classes $x^G$ of $G$
the group $G$ is generated by a Sylow $p$-subgroup together with a
conjugate of $x$.
The upper bound $[\N_G(P):P]$ for $|\N_G(P)|/|\N_M(P)|$,
for a maximal subgroup $M$ of $G$ that contains $P$,
is not good enough in some of the cases considered here;
instead of it, we can compute $|\N_M(P)|$ from the character table of $M$
whenever $P$ is cyclic, or use information about a larger known subgroup
of $M$ in which $P$ is normal.
We get that except in a few cases, all nontrivial classes $x^G$ have the
desired property.
In the remaining cases, if $x$ is an $r$-element,
either one sees that some other $r$-element $y$ satisfies
$G = \langle P, y \rangle$,
or one can find a random element $y \in x^G$ that works.
(The latter happens exactly for $p = r = 3$
and when $G$ is one of the groups $M_{23}$ or $HS$.)

Now consider the case that $G$ is the Monster,
which is special because the complete list of classes of maximal subgroups
of $G$ is currently not known.
From \cite{NW13} and \cite{Mmaxes} we know $44$ classes of maximal subgroups,
and that each possible additional maximal subgroup is almost simple
and has socle $L_2(13)$, $U_3(4)$, $U_3(8)$, or $Sz(8)$.
This implies that we know all those maximal subgroups that contain
a Sylow-$p$-subgroup of $G$ except in the case $p = 19$,
where maximal subgroups with socle $U_3(8)$ may arise.

Thus let us first consider that at least one of $p$, $r$
is different from $19$.
In this situation,
we use the same approach as for the other sporadic simple groups.
The only complication is that not all permutation characters $1_M^G$,
for the relevant maximal subgroups $M$ of $G$, are known;
however, if this happens then the character table of $M$ is known,
and we can compute the possible permutation characters,
and take the common upper bounds for these characters.
In each case, we get that the claimed property holds.

Finally, let $p = r = 19$.
The group $G$ has exactly one class of elements of order $19$.
Let $x$ be such an element.
From the character table of $G$, we can compute that there exist
conjugates $y$ of $x$ such that $x y$ has order $71$.
Since $\langle x, y \rangle = \langle x, x y \rangle$ holds
and no maximal subgroup of $G$ has order divisible by $19 \cdot 71$,
we have $\langle x, y \rangle = G$.
\end{proof}  

% (With somewhat more work, one could prove the stronger property
% from Proposition~\ref{p:spor} for Sylow $p$-subgroups $P$,
% not just Sylow $2$-subgroups.)

\end{document}